\documentclass{amsart}
\makeatletter
 \def\LaTeX{\leavevmode L\raise.42ex
   \hbox{\kern-.3em\size{\sf@size}{0pt}\selectfont A}\kern-.15em\TeX}
\makeatother

\newcommand{\BibTeX}{{\rm B\kern-.05em{\sc
i\kern-.025emb}\kern-.08em\TeX}}
\newtheorem{thm}{Theorem}[section]
\newtheorem{lem}[thm]{Lemma}
\theoremstyle{definition}
\newtheorem{defn}{Definition}

\numberwithin{equation}{section}
\begin{document}
\title{A reconstruction method for  band-limited signals on the hyperbolic plane}

\author {Hans Feichtinger}
\address {Department of Mathematics, Vienna University, AUSTRIA }
\email{hans.feichtinger@univie.ac.at}

\author{Isaac Pesenson}
\address{Department of Mathematics, Temple University, Philadelphia, PA19122}
 \email{pesenson@math.temple.edu}
\keywords{Poincare upper half-plane, Helgason-Fourier transform,
convolutions, band-limited functions, Laplace-Beltrami operator}
\subjclass{ 42C05; Secondary 41A17, 41A65, 43A85, 46C99 }

\begin{abstract}
A notion of band limited functions is considered in the case of
the hyperbolic plane in its Poincare upper half-plane $\mathbb{H}$
realization. The  concept of band-limitedness is based on the
existence of the Helgason-Fourier transform on $\mathbb{H}$. An
iterative algorithm is presented, which allows to reconstruct
band-limited functions from some countable sets
 of their values.
 It is shown that for sufficiently dense metric
lattices a geometric rate of convergence can be guaranteed as long
as the sampling density is high enough compared to the band-width
of the sampled function.

\end{abstract}

\maketitle

\section{Introduction}

The main goal of the present article is to consider  an iterative
algorithm for reconstruction of band limited functions on
 the two-dimensional hyperbolic plane in its Poincare
upper half-plane realization.

The notion of band limited functions plays a central role in the
classical signal analysis in which signals propagate in Euclidean
space. It seems interesting  to extend this theory to other
geometries
 in particular to
hyperbolic spaces. In this connection we would like to mention
very interesting investigations  of A.~Kempf \cite{K1}, \cite{K2},
 who used our sampling theory on manifolds to
 develop an approach to  quantization of space-time and
information theory.

 Analysis on the hyperbolic plane is also useful for other
 applied problems.
As it was discovered by C. Berenstein and E. Casadio Tarabusi in
\cite{BT1}, analysis on the hyperbolic plane plays an important
role for electrical impedance imaging.
 Another interesting
applications of the analysis on the hyperbolic plane to microwave
technology were given by A.~Terras  in \cite{T}.

The classical sampling theorem says that if $f\in
L^{2}(\mathbb{R})$ is from the Paley-Wiener space $PW_{\omega}$, $\omega >0$,
i.e., if its Fourier transform $\hat{f}$ has support
$\operatorname{supp}( \hat f) $ in $[-\omega, \omega]$, then $f$
is completely determined by its values at points $n\Omega $, where
$\Omega =\pi /\omega $, i.e.,
$$ f(t)= \sum f(n\Omega )\frac{\sin(\omega (t-n\Omega ))}{\omega (t-n\Omega )},$$
where the equality holds in the $L_2$-sense. The Paley-Wiener theorem states that $f \in PW_{\omega}$ if and
only if  $f$ is an entire function of exponential type $\omega $.

The  irregular sampling theory which was started by Paley and
Wiener \cite{PW34} was further developed by A.~Beurling,
P.~Malliavin \cite{B64}, \cite{BM67}, and H.~Landau \cite{Lan67}.

In the early 1990's, H.~Feichtinger and K.~Gr\"ochenig have
introduced iterative  reconstruction  methods (cf., the series of
papers \cite{FG92}, \cite{FG92b}, \cite{FG94}) which allow to
recover band-limited functions on $\mathbb{R}^{d}$ from their
irregular samples.

 On the other hand, a version of an
irregular sampling theory in which band-limited functions are
reconstructed as limits of variational splines was developed by
I.~Pesenson in the case of $L_{2}(\mathbb{R}^{d})$ \cite{Pes99},
in the case of stratified Lie groups \cite{Pes98a}, \cite{Pes98b},
in the case of manifolds \cite{Pes95}, \cite{Pes00},
\cite{Pes04a}, \cite{Pes04b}, and in the case of a  general
Hilbert space \cite{Pes01}. Note that the notion of band-limited
functions in Hilbert spaces was developed by I.~Pesenson in
\cite{KPes90}, \cite{Pes89} and \cite{Pes90} in connection with
approximation theory and Besov spaces on manifolds.

In our recent paper \cite{FP04}, we developed an iterative
algorithm for reconstruction of band limited functions on general
Riemannian manifolds of bounded geometry. In the same paper we
also proposed to specify our results for  manifolds which are of
interest for applications.

The present paper is organized as follows. In the second section
we will briefly introduce Harmonic Analysis on the Poincare upper
half plane $\mathbb{H}=SL(2,\mathbb{R})/
 SO(2)$. In the third section we introduce our basic notion  of
band limited function by using a non-commutative Fourier analysis
on $\mathbb{H}$ which is given by the Helgason- Fourier transform.
In the last section we present our iterative reconstruction
algorithm.

It is important to note that, due to the existence of rich
harmonic analysis on $\mathbb{H}$, our description of band limited
functions and our reconstruction method become much more
constructive when compared to the case of a general Riemannian manifold.

 \section{The Poincare upper half-plane model of the hyperbolic plane}

  Let $G=SL(2,\mathbb{R})$ be the special linear group of all
$2\times 2$ real matrices with determinant 1 and let $K=SO(2)$ be the
group of all rotations of $\mathbb{R}^{2}$. The factor
$\mathbb{H}=G/K$ is known as the 2-dimensional hyperbolic space
and can be described in many different ways. In the present paper
we consider the realization of $\mathbb{H}$ which is called
Poincare upper half-plane (see \cite{Helg1}, \cite{T}).

As a Riemannian manifold, $\mathbb{H}$ is identified with the
regular upper half-plane of the complex plane
$$
\mathbb{H}=\{x+iy|x,y\in \mathbb{R}, y>0\}
$$
with a new  Riemannian metric
$$
ds^{2}=y^{-2}(dx^{2}+dy^{2})
$$
and corresponding Riemannian measure
 $$
  d\mu=y^{-2}dxdy.
   $$
   If we
define the action of $\sigma\in G$ on $z\in \mathbb{H}$ as a
fractional linear transformation
$$
\sigma\cdot z=(az+b)/(cz+d),
$$
then the metric $ds^{2}$ and the measure $d\mu$  are invariant
under the action of $G$ on $\mathbb{H}$. The point $i=\sqrt{-1}\in
\mathbb{H}$ is invariant for all $\sigma\in K$. The Haar measure
$dg$ on $G$ can be normalized in a way that the  following
important formula holds true
\begin{equation}
\int_{\mathbb{H}}f(z)y^{-2}dxdy=\int_{G}f(g\cdot i)dg.
\end{equation}

In the corresponding space of square integrable functions
$L_{2}(G)$ with the inner product
$$
<f,h>=\int_{\mathbb{H}}f\overline{h}y^{-2}dxdy,
$$
we consider the Laplace-Beltrami operator

 $$
 \Delta=y^{2}\left(\partial_{x}^{2}+\partial_{y}^{2}\right)
 $$
of the metric $ds^{2}$.

 It is known that as an operator
in $L_{2}(\mathbb{H})=L_{2}(\mathbb{H},d\mu)$ which is initially
defined on $C_{0}^{\infty}(\mathbb{H})$, $\Delta$ has a self-adjoint closure
in $L_{2}(\mathbb{H})$.

Moreover, if $f$ and $\Delta f$ belong to $L_{2}(\mathbb{H})$,
then
$$
<\Delta f,f>\leq -\frac{1}{4}\|f\|^{2},
$$
where $\|f\|$ denotes the $L_{2}(\mathbb{H})$ norm of $f$.

We can introduce the full scale of Sobolev spaces
$H^{\sigma}=H^{\sigma}(\mathbb{H}), \sigma\in \mathbb{R}.$ For
$\sigma>0$ the space $H^{\sigma}$ is understood as the domain of
the operator $(-\Delta)^{\sigma/2}$ in the sense of the general
theory of positive definite self-adjoint operators with the graph
norm
$$
\|f\|_{\sigma}=\|f\|+\|\Delta^{\sigma/2}f\|.
$$

 For $\sigma<0$, $H^\sigma$ is understood as a dual space to
$H^{-\sigma}$ with respect to the paring
$$
<f,h>=\int_{\mathbb{H}}fhd\mu.
$$

 The Helgason-Fourier transform
of $f$ is defined by
the formula
$$
\hat{f}(s,\varphi)=\int_{\mathbb{H}}f(z)\overline{Im(k_{\varphi}z)^{s}}
y^{-2}dxdy,
$$
for $s\in \mathbb{C}, \varphi\in (0,2\pi]$, and where $k_{\varphi}\in SO(2)$ is the rotation of $\mathbb{R}^{2}$
by angle $\varphi$.

We have the following inversion formula for all functions $f$ from
the space $ C_{0}^{\infty}(\mathbb{H})$ of infinitely
differentiable functions on $\mathbb{H}$ with compact support:
$$
f(z)=(8\pi^{2})^{-1} \int_{t\in\mathbb{R}}\int_{0}^{2\pi}
\hat{f}(it+1/2,\varphi)Im(k_{\varphi}z)^{it+1/2}t \tanh \pi t
d\varphi dt.
$$

The Plancherel Theorem states that a map
$f\rightarrow \hat{f}$ can be extended to an isometry of
$L_{2}(\mathbb{H})$ (with respect to invariant measure $d\mu $) onto
$L_{2}(\mathbb{R}\times (0,2\pi ])$ (with respect to the measure
$(8\pi^{2})^{-1} t \tanh \pi t dt d\varphi$).

If $f$ is a function on $\mathbb{H}$ and $\varphi$ is a
$K=SO(2)$-invariant function on $\mathbb{H}$ their convolution is
defined by the formula
$$
f\ast \varphi(g\cdot i)=\int_{SL(2,\mathbb{R})}f(gu^{-1}\cdot
i)\varphi(u)du, i=\sqrt{-1},
$$
where $du$ is the Haar measure on $SL(2,\mathbb{R})$. It is known,
that for the Helgason-Fourier transform the following formula
holds true
\begin{equation}
\widehat{f\ast\varphi}=\hat{f}\cdot\hat{\varphi}.
\end{equation}

\section{Band limited signals and metric lattices}

In this section we introduce notions  of band limited signals,
metric lattices, develop some of their properties and prove
inequalities which are used later.
\begin{defn}
 We will say that $f\in L_{2}(\mathbb{H})$ is  $\omega$-band limited
 function if $\hat{f}(s,\varphi)=0$, for
$|s|>\omega$. The set of all $\omega$-band limited functions will
be denoted as $E_{\omega}(\mathbb{H})$.
\end{defn}
According to this variant of Harmonic Analysis on $\mathbb{H}$, the
following formula holds:
\begin{equation}
\widehat{\Delta f}=-\left(s^{2}+\frac{1}{4}\right)\hat{f}.
\end{equation}
This  formula, the Plancherel theorem, and the Sobolev embedding
theorem immediately imply the next result.
\begin{thm}
For any real $\sigma$ and every $\omega\geq 0$ the set
$E_{\omega}(\mathbb{H})$ is a closed invariant linear subspace of
the Sobolev space $\textsl{H}^{\sigma}(\mathbb{H})$. If $f\in
E_{\omega}(\mathbb{H})$ then, for every $\sigma\geq 0$, the function
$\Delta^{\sigma} f$ belongs to $C^{\infty}(\mathbb{H})$ and is
bounded on $\mathbb{H}$.

\end{thm}

Moreover, the same formula (3.1) and the Plancherel theorem give
us the following inequality (3.2), which   is appropriate to call
the Bernstein inequality.
\begin{thm}
A function  $f$ belongs to the space $E_{\omega}(\mathbb{H}),
\omega \ge 0$, if and only if for  every $\sigma\in \mathbb{R}$,
the following inequality holds true
\begin{equation} \|\Delta^{\sigma} f\|\leq
\left(\omega^{2}+\frac{1}{4}\right)^{\sigma}\|f\|.
 \end{equation}
 \end{thm}

\begin{proof}
 By using the Plancherel formula and (3.1) we obtain that for every $\omega$-
  band limited function
  $$
\|\Delta^{\sigma}f\|^{2}=
\frac{1}{8\pi^{2}}\int_{|t|<\omega}\int_{0}^{2\pi}\left(t^{2}+\frac{1}{4}\right)
^{\sigma} |\widehat{f}(t,\varphi)|^{2}t\tanh \pi t d\varphi dt
\leq \left(\omega^{2}+\frac{1}{4}\right)^{\sigma}\|f\|^{2}.
$$

  Conversely, if $f$ satisfies (3.2), then for any $\varepsilon>0$
  and any $\sigma >0$ we have
  $$
  \frac{1}{8\pi^{2}}\int_{|t|<\omega+\varepsilon}\int_{0}^{2\pi}
  |\hat{f}(t,\varphi)|^{2}t\tanh \pi t d\varphi dt
  \leq
  $$
  $$
  \frac{1}{8\pi^{2}}\int_{|t|<\omega+\varepsilon}\int_{0}^{2\pi}
  \left(t^{2}+\frac{1}{4}\right)^{-2\sigma}
 \left( t^{2}+\frac{1}{4}\right)^{2\sigma}|\hat{f}(t,\varphi)|^{2}t\tanh \pi t d\varphi dt
 \leq
  $$
  \begin{equation}
  \left(\frac{\omega^{2}+\frac{1}{4}}
  {(\omega+\varepsilon)^{2}+\frac{1}{4}}\right)^{2\sigma}
  \|f\|^{2}.
  \end{equation}
It means that, for any $\varepsilon>0$, the function
$\widehat{f}(t,\varphi)$ is zero on
$(-\omega-\varepsilon,\omega+\varepsilon)\times (0,2\pi]$. The
statement is proved.
\end{proof}
We take the point $i\in \mathbb{H}$ and consider a small open ball
$B(i,r/4)$ in the invariant metric $ds^{2}=y^{-2}(dx^{2}+dy^{2})$.
It is possible to find such elements $g_{j}\in G$ that the family
of balls $B(x_{j},r/4), x_{j}=g_{j}\cdot i,$ has the following
maximal property: there is no ball in $\mathbb{H}$ of radius $r/4$
which would have empty intersection with every ball from this
family. Then the balls of double radius $B(x_{j},r/2)$ would form
a cover of $\mathbb{H}$. Of course, the balls $B(x_{j},r)$ will
also form a cover of $\mathbb{H}$. Let us estimate the
multiplicity of this cover.

Note that the Riemannian volume $B(\rho)$ of a ball in
$\mathbb{H}$ is independent of its  center and   is given by the
formula
$$
B(\rho)=2\pi\int_{0}^{\rho}\sinh t dt.
$$

 Every ball from the family $\{B(x_{\nu}, r)\}$ that has
non-empty intersection with a particular ball $B(x_{\mu}, r)$ is
contained in the ball $B(x_{\mu}, 3r)$. Since any two balls from
the family $\{B(x_{\nu}, r/4)\}$ are disjoint, it gives the
following estimate for the index of multiplicity  of the cover
$\{B(x_{\nu}, r)\}$:
\begin{equation}
\frac{B(3r)}{B( r/4)}\leq\frac{\int_{0}^{3r}\sinh t
dt}{\int_{0}^{r/4} \sinh t dt}\leq (12)^{2}e^{3r}.
\end{equation}

In fact, the last estimate can be improved.

  We have proved the following Lemma.
\begin{lem}
For any sufficiently  small  $r>0$ there
 exists a set of points $\{x_{j}\} \subset \mathbb{H}$, such that:

\bigskip

1) open balls $B(x_{j}, r/4)$ are disjoint;

\bigskip

 2) open balls $B(x_{j},r/2)$ form a cover of $\mathbb{H}$;

 \bigskip

3) multiplicity of the cover by open balls $B(x_{j}, r)$ is not
greater $N=N(\mathbb{H})=12^{2}e^{3}.$

\bigskip

\end{lem}

\begin{defn}
We will use notation $Z(\{x_{j}\}, r, N)$ for any set of points
$\{x_{j}\}\in \mathbb{H}$ which satisfies the properties 1)- 3)
from the  Lemma 3.3  with a positive small $r$ and we will call
such set a metric $(r,N)$-lattice of $\mathbb{H}$.
\end{defn}

Let us emphasize that an $(r,N)$-lattice assumes a "uniform"
distribution of points on $\mathbb{H}$ with respect to the
\textit{hyperbolic metric}. It will not look as a "uniform"
distribution from the point of view of the Euclidean metric: it
will become denser and denser when  approaching  the real line and
 will become sparser and sparser when going to infinity.

Sobolev spaces $H^{\sigma}(\mathbb{H}),\sigma\in \mathbb{R},$ can
also be described in terms of local geodesic coordinates on
$\mathbb{H}$.

In what follows we fix a  small positive $r >0$ and an $(r,
N)$-lattice in $\mathbb{H}$ which will be denoted by $Z(\{y_{\nu}\},
r, N)$. For the corresponding set of balls $\{B(y_{\nu},r)\}$,
which satisfy Lemma 3.3,  we consider \textit{a uniformly bounded
partition of unity $\psi=\{\psi_{\nu}\}$ associated with
$\{B(y_{\nu},r/2)\}$}. Namely, we give the following Definition.

\begin{defn} A uniformly bounded
partition of unity $\psi=\{\psi_{\nu}\}$ associated with
$\{B(y_{\nu},r/2)\}$ is a set of non-negative
$C^{\infty}(\mathbb{H})$ functions such that

\bigskip

a) $\operatorname{supp}\psi_{\nu}\subset B(y_{\nu},r/2),$

\bigskip

 b)
$\sup_{x}|\partial^{|\alpha|}\psi_{\nu}|\leq C(\alpha),$ where
$C(\alpha)$ is independent of $\nu.$

\bigskip

\end{defn}

 An equivalent norm on $H^{\sigma}$ can be introduced by the formula
\begin{equation}
\|f\|_{\sigma} \simeq \left(\sum_{\nu}\|\psi_{\nu}
f\|^{2}_{H^{\sigma}(\mathbb{R}^{2})}\right)^{1/2},
\end{equation}
where $\|f\|_{H^{\sigma}(\mathbb{R}^{2})}$ denotes the regular Sobolev norm on
the plane $\mathbb{R}^{2}$.

The following important result is an adaptation of the Lemma 3.3
from \cite{Pes00} for the case when the dimension of the manifold
is $d=2$.

\begin{thm}
For any $k>1$  there exist a constant $C=C(k, N)>0 $
 such that for any sufficiently small $r>0$ and any $(r,N)$-lattice
 $Z=Z(\{x_{\mu}\},r,N)$ the following inequality holds true

\begin{equation}
\|f\|\leq C\left\{r\left(\sum_{x_{j}\in Z}
|f(x_{j})|^{2}\right)^{1/2}+r^{k}\|\Delta^{k/2}f\|\right\}, \;\; k>1.
\end{equation}

\end{thm}

\section{ Iterative reconstruction algorithm}

In this section we describe our first iterative algorithms for
reconstruction of band-limited functions on manifolds. The
following Lemma 4.1 presents a generic idea which is used for
reconstruction of different classes of analytic functions. This
result basically says that  if a  bounded operator is close to
identity operator then it can be inverted by a Neumann series.

The goal of section is to realize this idea in the case of
band-limited functions on manifolds.

\begin{lem}
Let $A$ be a bounded operator in a Hilbert space $H$ such that for
some $\gamma <1$ and for all $f \in H$ \begin{equation}
\|f-Af\|\leq \gamma \|f\|. \end{equation} Then $A$ is invertible
and $f$ can be recovered from $Af$ by the following iterative
procedure. If $f_{0}=Af$ and
\begin{equation} f_{n+1}=f_{n}+A(f-f_{n})
\end{equation}
then
\begin{equation}
\lim_{n\rightarrow\infty}f_{n}=f
\end{equation}
with the error estimate
\begin{equation}
\|f-f_{n}\|\leq\gamma^{n+1}\|f\|.
\end{equation}
\end{lem}

\begin{proof}
Since the norm of the operator $I-A$ is less than $1$, operator
$A$ is invertible and, thus, its inverse $A^{-1}$ can be represented by a
Neumann series $ A^{-1}=\sum_{n=0}(I-A)^{n}. $ Thus, every $f\in
H$ is determined by $Af$, because \begin{equation} f=A^{-1}Af\sum^\infty_{n=0}(I-A)^{k}Af  \lim_{n\rightarrow\infty} \left ( \sum^n_{n=0}(I-A)^{n}Af  \right ).
\end{equation}
The sequence $f_n$ of partial sums, $ f_{n} :\sum_{k=0}^{n}(I-A)^{k}Af $, satisfies the stated recursion
relation. Indeed,
$$f_{n+1} \sum_{k=0}^{n+1}(I-A)^{k}Af=Af+\sum_{k=1}^{n+1}(I-A)Af $$
 $$
Af+(I-A)\sum_{k=0}^{n}(I-A)^{k}Af=Af+(I-A)f_{n}=f_{n}+A(f-f_{n}).
$$
Moreover, the identity
$$
\sum_{k=n+1}^{\infty}(I-A)^{k}=(I-A)^{n+1}A^{-1},
$$
implies that for all $f \in E_{\omega}(M)$ we have
$$
\|f-f_{n}\|=\|\sum_{k=n+1}^{\infty}(I-A)^{k}Af\|
\|(I-A)^{n+1}A^{-1}Af\| \ \leq \ \gamma^{n+1}\|f\|.
$$

\end{proof}

We will apply the above Lemma in the following situation.

Recall that in the previous section in the Definition 3 we fixed a
uniformly bounded partition of unity $\{\psi_{\nu}\}$ subordinated
to a family of balls $\{B(y_{\nu}, r/2)\}$ and introduced Sobolev
norms by the formula (3.5).

 Consider an $(\varepsilon,N)$-lattice $Z( \{x_{j}\},
\varepsilon,N), \varepsilon\leq r/2,$ (see Definition 2). For an
open cover by balls $\{B(x_{j}, \varepsilon/2)\}$ satisfying
Definition 2, we construct a uniformly bounded partition of unity
$\{\theta_{j}\}$ subordinated to $\{B(x_{j}, \varepsilon/2)\}$,
with the properties:

\bigskip

a) $\operatorname{supp}\theta_{j}\subset B(x_{j},\varepsilon/2),$

\bigskip

 b)
$\sup_{x}|\partial^{|\alpha|}\theta_{j}|\leq C(\alpha),$ where
$C(\alpha)$ is independent of $j.$

\bigskip

Note that the assumptions  that $\varepsilon\leq r/2$ and that the
multiplicity of the cover $\{B(y_{\nu},r)\}$, which was used in
the definition of the Sobolev norm (3.5),  is not greater than
$N=N(\mathbb{H})$, imply that each ball $B(x_{j},\varepsilon/2)$
has non-empty intersections with no more than $N(\mathbb{H})$
balls of the family $\{B(y_{\nu},r/2)\}$.

 Given a function $f\in
H^{k}(\mathbb{H}), k>1,$ we consider an operator
\begin{equation}
V_{Z,\theta}(f)=\sum_{j}f(x_{j})\theta_{j}, \theta_{j}\in
C_{0}^{\infty}(B(x_{j},\varepsilon/2)).
\end{equation}
It will be shown that for any lattice $Z=Z(\{x_{\nu}\},
\varepsilon,N), \varepsilon\leq r/2$, and corresponding uniformly
bounded partition of unity $\theta=\{\theta_{j}\}$ subordinated to
the cover $\{B(x_{j}, \varepsilon/2)\}$, the function
$V_{Z,\theta}(f)$ belongs to $L_{2}(\mathbb{H})$ as long as $f$
belongs to $H^{k}(\mathbb{H})$, where $k>1$.

To construct the operator $A$ we will need an orthogonal projection
from $L_{2}(\mathbb{H})$  on the space of $\omega$-band limited
functions $E_{\omega}(\mathbb{H})$.

We are going to describe this projection in terms of the
Helgason-Fourier transform. In order to do so, we have to introduce an
analog of the classical $sinc$ function,
$$
sinc (t)=\frac {sin\pi t}{\pi t}.
$$
\begin{defn}
We definee the hyperbolic $sinch_{\omega}(g),
g\in SL(2,\mathbb{R})$, to be such a $SO(2)$-biinvariant function on
$SL(2,\mathbb{R})$ whose Helgason-Fourier transform
$$
\widehat{sinch_{\omega}}(s)
$$
is $1$ for $|s|\leq \omega$, and $0$ for $|s|>\omega$.
\end{defn}

Now, we introduce the operator $P_{\omega}$ by the formula
\begin{equation}
P_{\omega}f= f\ast sinch_{\omega},   f\in L_{2}(\mathbb{H}).
\end{equation}

It is clear that the operator
$$
P_{\omega}:L_{2}(\mathbb{H})\rightarrow  E_{\omega}(\mathbb{H})
$$
is the orthogonal projection on the subspace of $\omega$-band
limited functions.

 The operator $A$ will be defined as
\begin{equation}
A_{Z,\theta}f=P_{\omega} V_{Z,\theta}(f),
\end{equation}
where $V_{Z,\theta}$ is defined in (4.6) and $P_{\omega}$ is the
orthogonal projection from $L_{2}(\mathbb{H})$ onto
$E_{\omega}(\mathbb{H})$.

 The next
goal is to provide a uniform estimate for the norms of the
operator $I - A_{Z, \theta}$ on the subspace $E_{\omega}(\mathbb{H})$ for
all $(\varepsilon,N)$-lattices with $\varepsilon\leq r/2.$

The following lemma provides a step in this direction.
 \begin{lem}
For any lattice $Z=Z(\{x_{j}\},\varepsilon ,N), \varepsilon<r/2,$
and any uniformly bounded partition of unity
$\theta=\{\theta_{j}\}$ subordinated to $B(x_{j},\varepsilon/2)$,
the map $V_{Z,\theta}(f)$ is a continuous operator from
$H^{k}(\mathbb{H}), k> 1,$ into $L_{2}(\mathbb{H})$. In other
words, there exists a constant $C=C(\mathbb{H},k)$, such that
\begin{equation} \|V_{Z,\theta}(f)\|_{L_{2}(\mathbb{H})}\leq
C\|f\|_{H^{k}(\mathbb{H})},
\end{equation}
  for all $f\in H^{k}(\mathbb{H})$.

 \end{lem}
\begin{proof}

 According to (3.5) we have
$$
 \|V_{Z,\theta}(f)\|^{2}_{L_{2}(\mathbb{H})}\sum_{\nu}\|\psi_{\nu}V_{Z,\theta}f\|^{2}_{L_{2}(B(y_{\nu},r))}
$$
and
$$
\|\psi_{\nu}V_{Z,\theta}f\|^{2}_{L_{2}(B(y_{\nu},r))}\leq
\|\psi_{\nu}\sum_{j}f(x_{j})\theta_{j}\|^{2}_{L_{2}(B(y_{\nu},r))}\leq
C\varepsilon\sum_{j}|f(x_{j})|^{2},
$$
where supp $\theta_{j}$ is in $B(x_{j},\varepsilon/2)$ and $C$
depends on the multiplicity $N=N(\mathbb{H})$. In other words,
\begin{equation}
 \|V_{Z,\theta}(f)\|_{L_{2}(\mathbb{H})}\leq C\varepsilon^{1/2}
 \left(\sum_{j}|f(x_{j})|^{2}\right)^{1/2}.
 \end{equation}

 On the other hand, by a known inequality for $\mathbb{R}^{2}$,

$$
|f(y)|\leq C\sum_{0\leq m\leq
k}\varepsilon^{m-1}\|f\|_{H^{m}(B(x_{j},\varepsilon))}, k>1,
C=C(k),$$ where $y\in B(x_{j},\varepsilon/2), f \in
C^{\infty}(B(x_{j},\varepsilon))$, we obtain
$$
|f(x_{j})|\leq C\sup_{x\in B(x_{j}, \varepsilon/2)}|f(x)|\leq
C\varepsilon^{-1}\|f\|_{H^{k}(B(x_{j},\varepsilon))},
$$
where $k>1, C=C(k),$ and then

$$
\left(\sum_{j}|f(x_{j})|^{2}\right)^{1/2}\leq
C\varepsilon^{-1}\left(\sum_{j}\|f\|^{2}_{H^{k}(B(x_{j},\varepsilon))}\right)^{1/2}.
$$

The last inequality and the inequality (4.10) give that there
exists a constant $C$ which depends on smoothness $k$ and
multiplicity $N(\mathbb{H})$, for which (4.9) holds true.

\end{proof}

 We will need the following Lemma 4.3 from \cite{FP04}.
\begin{lem}
For any $k>1$ there exist a constant $C=C(\mathbb{H},k)>0, $ such
that for any small  $0<\varepsilon<r/2$ and any
$Z(\{x_{j}\},\varepsilon,N)$ the following inequality holds true

\begin{equation}
\sum_{j}\|f-f(x_{j})\|_{B(x_{j},\varepsilon/2)} \ \leq \ C
\varepsilon \|(I+\Delta)^{k/2}f\|, k>1.
 \end{equation}

\end{lem}

\begin{thm} For a given $k>1$
there exist a constant $C=C(\mathbb{H},k)>0$
 such that for any
lattice $Z(\{x_{j}\}, \varepsilon, N),$ with sufficiently small
$\varepsilon>0$ and for any $f\in E_{\omega}(\mathbb{H})$
\begin{equation}
\|f-A_{Z,\theta}f\| \ \leq  \ C \ \varepsilon(1+\omega^{2})^{k/2}\
\|f\|.
\end{equation}
Consequently, for fixed $k>1$ and $\omega>0$ it is true that for
any $\varepsilon>0$ satisfying
$$ \varepsilon<
\left(C(1+\omega^{2})^{k/2}\right)^{-1},
$$
one has
$$
\|f-A_{Z,\theta}f\| \  \leq \ \gamma \|f\| \quad \mbox{ where }
\quad \gamma=C\varepsilon(1+\omega^{2})^{k/2}<1 .
$$
\end{thm}
\begin{proof}

For any $f\in E_{\omega}(M)$ we have
$$
\|f-A_{Z,\theta}f\|=\|P_{\omega}f-P_{\omega}V_{Z,\theta}(f)\|\leq
\|f-V_{Z,\theta}(f)\|.
$$
Next, since   the following identity holds true $$
f(x)=\sum_{j=j_{1}}^{j_{N_{M}}}\theta_{j}(x)f(x),
$$
we obtain
$$ \|f-V_{X}(f)\|\leq
\|f-\sum_{j}\theta_{j}f(x_{j})\|\leq
\|\sum_{j}\theta_{j}f-\sum_{j}\theta_{j}f(x_{j})\|\leq
$$
$$
\sum_{j}\|f-f(x_{j})\|_{L_{2}(B(x_{j},\varepsilon/2)},
 $$
where the  sum  can be estimated by using our inequality (4.11).

  We have
$$
\|f-A_{Z,\theta}f\|\leq C\varepsilon\|(I+\Delta)^{k/2}f\|,
$$
for any $k>1$.

 Because, for $f\in E_{\omega}(\mathbb{H})$, the Bernstein
inequality
$$
\|\Delta^{k}f\|\leq \omega^{2k}\|f\|
$$
holds true, it yields the inequality
$$
\|f-A_{Z,\theta}f\|\leq C\varepsilon (1+\omega^{2})^{k/2}\|f\|.
$$

 For a fixed
$\omega>0$ and $k>1$, because $C=C(\mathbb{H}, k)$ depends only
on $\mathbb{H}$ and $k$, we can choose
$$
\varepsilon< \left(C(1+\omega^{2})^{k/2}\right)^{-1},
$$
such that corresponding
$$
\gamma=C\varepsilon(1+\omega^{2})^{k/2}
$$
is less than 1. Theorem is proved.
\end{proof}

Combining this result with  Lemma 4.1 we obtain
 \begin{thm} For a given $k>1$ and
$\omega>0$ choose an $\varepsilon>0$ such that $$
\varepsilon<(C)^{-1}(1+\omega^{2})^{-k/2}, $$ where the constant
$C=C(\mathbb{H},k)$ taken from Theorem 4.4. Then, for any
$(\varepsilon, N)$-lattice $Z(\{x_{j}\},\varepsilon, N)$ and for
any corresponding operator $A=A_{Z,\theta}$ in the space
$E_{\omega}(\mathbb{H})$ define inductively, starting from

$f_{0}=Af$:
$$
f_{n+1}=f_{n}+A(f-f_{n}).
$$
Then in the space $L_{2}(\mathbb{H})$ the following convergence
holds true
$$ \lim_{n\rightarrow\infty}f_{n}=f \ \mbox{in } \  L_{2}(\mathbb{H}),
$$
with the error estimate $$
 \ \mbox{and} \quad
\|f-f_{n}\|\leq\gamma^{n+1}\|f\|, $$ where
$\gamma=C\varepsilon(1+\omega^{2})^{k/2}<1.$

\end{thm}

\end{document}